\documentclass[onefignum,onetabnum]{siamart171218}

\usepackage{cite}
\usepackage{lipsum}
\usepackage{tikz}
\usetikzlibrary{shapes}
\usetikzlibrary{backgrounds}
\usepackage{subfig,subfloat}
\usepackage{amsfonts}
\usepackage{graphicx}
\usepackage{epstopdf}
\usepackage[algo2e]{algorithm2e}
\ifpdf
\DeclareGraphicsExtensions{.eps,.pdf,.png,.jpg}
\else
\DeclareGraphicsExtensions{.eps}
\fi

\newsiamremark{remark}{Remark}
\newsiamthm{claim}{Claim}

\headers{On Colourability of Polygon Visibility Graphs}{O.~\c{C}a\u{g}irici, P.~Hlin\v en\'y, B. Roy}

\title{On Colourability of Polygon Visibility Graphs%
\thanks{{P.~Hlin\v en\'y and O.~\c{C}a\u{g}irici have been
		supported by the {Czech Science Foundation}, project no. {17-00837S}.}} }

\author{Onur \c{C}a\u{g}irici%
	\thanks{Department of Computer Science, Masaryk University, Brno, CZ 
		(\email{onur@mail.muni.cz}, \email{hlineny@fi.muni.cz},
			\email{bodhayan.roy@gmail.com}).}
	\and Petr Hlin\v en\'y\footnotemark[2]
	\and Bodhayan Roy\footnotemark[2]
	\thanks{Current affiliation: Indian Institute of Technology Kharagpur, India}
}

\usepackage{amsopn}

\ifpdf
\hypersetup{
	pdftitle={Colourability of Polygon Visibility Graphs},
	pdfauthor={O.~Cagirici, P.~Hlineny, B.~Roy}
}
\fi

\begin{document}
	
	\maketitle
	
	\begin{abstract}
		We study the problem of colouring visibility graphs of polygons.
		In particular, for visibility graphs of simple polygons,
		we provide a polynomial algorithm for $4$-colouring, and prove that the
		$5$-colourability question is already NP-complete for them.
		For visibility graphs of polygons with holes, we prove that the $4$-colourability question is NP-complete.
		
	\end{abstract}
	
	

	\section{Introduction} \label{sec:intro}
	
	Visibility graphs are widely studied graph classes in computational
	geometry.  Geometric sets such as sets of points or line segments, polygons, polygons with
	obstacles, etc., all can correspond to specific visibility graphs, and have
	uses in robotics, signal processing, security paradigms, decomposing shapes into clusters
	\cite{Aichholzer2011ConvexifyingPW,AIGNER19841,Latombe:1991:RMP:532147,g-vap-07,o-agta-87}. 
	We study the visibility graphs of simple polygons in the Euclidean
	plane, but we also mention polygons with (again polygonal) holes in
	Section~\ref{sec:polygonswithholes}.
	To make things clear, all polygons in the paper are simple unless stated otherwise.
	
	Given an $n$-vertex polygon $P$ (not necessarily convex) in the
	plane, two points $p$ and $q$ of $P$ are said to be \emph{mutually visible} if,
	and only if the line segment $\overline{pq}$ does not intersect the
	exterior of $P$.  The $n$-vertex visibility graph $G(V,E)$ of $P$ is defined as
	follows.  The vertex set $V$ of $G$ contains a vertex $v_i$ if, and only if,
	the polygon $P$ contains the point $p_i$ as its vertex.
	The edge set $E$ of $G$ contains an edge $\{v_i,v_j\}$ if, and only if, 
	the points $p_i$ and $p_j$ are mutually visible.
	Given a polygon $P$ in the plane, we can compute its visibility graph $G$ in $\mathcal{O}(n^2)$
	time using the polygon triangulation method \cite{OROURKE1998105,Hershberger89}. 
	Hence, in this paper, we slightly abuse notation by not distinguishing
	between a polygon $P$ and its visibility graph $G$ and referring to a
	polygon vertex $p_i$ as to the corresponding $G$-vertex~$v_i$.

	Visibility graphs of polygons have been studied with respect to various theoretical and practical computational problems.
	The complexities of several popular optimization problems
	have been determined for visibility graphs of polygons.
	A geometric variation of the dominating set problem, namely polygon guarding, is one of the most studied problems
	in computational geometry and is known as the Art Gallery Problem \cite{o-agta-87}. 
	It has been studied extensively
	for both polygons with and without holes and has been found to be NP-hard in both cases \cite{ll-ccagp-86,os-snpd-83}.
	Besides, given a polygon, computing a maximum independent set is known to be hard, 
	due to Shermer \cite{s-hpip-89} (see also~\cite{ls-cavg-95} for	other problems),
	while computing a maximum clique has been shown to be in polynomial time by Ghosh et al.~\cite{gsbg-cmcvgsp-07}.
	
	A \emph{proper vertex colouring} of a graph is an assignment of labels or colours to the vertices of
	the graph so that no two adjacent vertices have the same colours. Henceforth, 
	when we say colouring a graph, we refer to proper vertex colouring.
	The \emph{chromatic number} of a graph is defined as the minimum number of colours used in any proper colouring of the graph.
	Visibility graph colouring has been studied for various types of visibility graphs.
	Babbitt et al. gave upper bounds for the chromatic numbers of k-visibility graphs of arcs and segments \cite{JGAA-362}.
	K\'ara et al.\ characterized 3-colourable visibility graphs of point sets and
	described a super-polynomial lower bound on the chromatic number with respect to
	the clique number of visibility graphs of point sets \cite{kpw-ocnv-2005}.
	Pfender showed that, as for general graphs, the chromatic number of visibility graphs 
	of point sets is also not upper-bounded by their clique numbers \cite{p-vgps-2008}.
	Diwan and Roy showed that for visibility graphs of point sets, the 5-colouring
	problem is NP-hard, but 4-colouring is solvable in polynomial time \cite{pvg_col}.
	
	The problem of {\em colouring} the visibility graphs of given polygons 
	has been studied in the special context where each internal point of the
	polygon is seen by a vertex, whose colour appears exactly
	once among the vertices visible to that point \cite{Bartschi:2014,hoffmann:2015,FeketeFHM014}. 
	However, little is known on colouring visibility graphs of polygons without such constraints.
	Although 3-colouring is NP-hard for general graphs \cite{cai-79}, in particular it is rather trivial
	to solve it for visibility graphs of polygons in polynomial time using a greedy approach. 
	With $4$ colours the same question has been open so far
	(precisely, until the conference paper~\cite{DBLP:conf/fsttcs/CagiriciHR17}).
	
	In this paper we completely settle the complexity question of
	the general problem of colouring polygonal visibility graphs,
	which was declared open in 1995 by Lin and Skiena \cite{ls-cavg-95}.
	In Section \ref{sec1}, we provide a polynomial-time algorithm to find a $4$-colouring
	of a given graph $G$ with the promise that $G$ is the visibility graph of some polygon, if $G$ is indeed $4$-colourable.
	On the other hand, in Section \ref{sec:hardness5} we provide a reduction showing that the question of
	$k$-colourability of the visibility graph of a given simple polygon 
	is NP-complete for any~$k\geq5$.
	We remark that in the conference version of this
	paper~\cite{DBLP:conf/fsttcs/CagiriciHR17} we used a different
	reduction showing hardness only for~$k\geq6$.
	In Section~\ref{sec:polygonswithholes}, we additionally show that
	already the question of $4$-colourability 
	of visibility graphs of polygons with holes is an NP-complete problem.
	
	\section{4-Colouring visibility graphs} \label{sec1}
	
	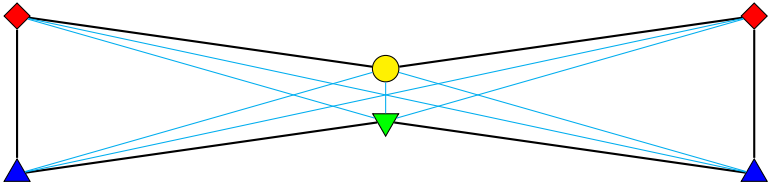
\begin{figure}
		\centering
		\begin{tikzpicture}[scale=0.7,
			triangle/.style = {regular polygon, regular polygon sides=3},
			dtriangle/.style = {regular polygon, regular polygon sides=3, rotate=180}
		]

		\tikzstyle{every node}=[draw, fill=white, shape=circle, minimum size=10pt, inner sep=2pt];
		\node[fill=blue,triangle] (A) at (-5,0) {};
		\node[fill=red,diamond] (B) at (-5,3) {};
		\node[fill=yellow] (C) at (2,2) {};
		\node[fill=red,diamond] (D) at (9,3) {};
		\node[fill=blue,triangle] (E) at (9,0) {};
		\node[fill=green,dtriangle] (F) at (2,1) {};
		
		\draw[thick] (A)--(B)--(C)--(D)--(E)--(F)--(A);
		\tikzstyle{every path}=[color=cyan];
		\draw (A)--(C);
		\draw (A)--(D);
		\draw (B)--(E);
		\draw (B)--(F);
		\draw (C)--(E);
		\draw (C)--(F);
		\draw (D)--(F);
		
		\end{tikzpicture}
		\caption{A visibility graph that is non-planar, but is 4-colourable.
		To improve readability of a possible monochromatic variant of the
		picture, we also indicate different colours by different node shapes.}
		\label{fig:nonplanar}
	\end{figure}
	
	In this section, we study the algorithmic question of $4$-colourability of the
	visibility graph of a given polygon.
	The full structure of $4$-colourable visibility graphs is not yet known
	and it seems to be non-trivial. 
	For instance, if a visibility graph is planar, it is obviously $4$-colourable.
	Though, if such a graph contains $K_5$, then it is neither planar nor $4$-colourable, 
	but a visibility graph not containing any $K_5$ may be non-planar yet $4$-colourable
	(Figure \ref{fig:nonplanar}).
	
	The related algorithmic problem of 3-colouring visibility graphs is rather
	easy to resolve as follows.
	Every simple polygon can be triangulated and, in such a triangulation, 
	every non-boundary edge is contained in two triangles.
	One can then proceed greedily edge by edge: 
	Suppose a triangle has already been coloured, and it
	shares an edge with a triangle that is not fully coloured. 
	Then the two end vertices of the shared edge 
	uniquely determine the colour of the third vertex of the uncoloured triangle. 
	
	Our algorithm essentially generalizes the 3-colouring method for 4-colouring.
	We first divide the polygon into \emph{reduced polygons}.
	A polygon $P$ is called a reduced polygon, if every chord of $P$
	(i.e., an internal diagonal) is intersected by another chord of $P$.
	After the division, we find and colour in each reduced subpolygon a triangle 
	(a $K_3$ subgraph) with three distinct colours.
	Subsequently, whenever we find an uncoloured vertex $v$ adjacent to some three
	vertices coloured with three distinct colours (such as, to an already
	coloured triangle), we can uniquely colour also $v$, by the fourth colour.
	We will show that we can exhaust all vertices of a reduced subpolygon in this manner.
	Furthermore, we check for possible colouring conflicts -- since the
	colouring process is unique, this suffices to solve $4$-colourability.
	
	Altogether, this will lead to the following theorem.
	\begin{theorem} \label{colthm}
		The 4-colourability problem is solvable in polynomial time for visibility graphs of simple polygons,
		and if a 4-colouring exists, then it can be computed in polynomial time
		from the given input graph (even without a visibility representation).
	\end{theorem}
	
	\subsection{Unique 4-colouring of reduced polygons}
	We first prove that if a reduced polygon is 4-colourable, then the 4-colouring is unique up to
	a permutation of colours. In the coming proof,
	consider a polygon $P$ and its visibility graph $G(V,E)$, embedded on $P$.
	Hereafter we slightly abuse notation by equating $P$ and $G$.
	Since we want to 4-colour $P$, we assume that $G$ has no $K_5$ (or we answer `no').
	We denote the clockwise polygonal chain of $P$ from a vertex $u$ to
	a vertex $v$ as $\Gamma(u,v)$.
	
	One can easily see that it is enough to focus on reduced $P$ in our proofs.
	Indeed, assume an edge $uv$ of $G$ which is a chord of $P$ and not crossed
	by any other chord.
	We can partition $P$ into subpolygons $P_1$ and $P_2$, where
	$P_1=(u\,\Gamma(u,v)\,v)$ and $P_2= (v\,\Gamma(v,u)\,u)$. 
	Since no edge of $G$ has one end in $P_1\setminus P_2$ and the other in 
	$P_2\setminus P_1$, the polygons $P_1$ and $P_2$ can be $4$-coloured separately
	and merged again (provided that $P_1$ and $P_2$ are $4$-colourable).
	
	Let $u$ and $v$ be two vertices of $P$.
	The {\em shortest path} between $u$ and $v$ is a (graph) path from $u$ to
	$v$ in $G$ such that the sum of the Euclidean lengths of its edges is minimized.
	Such a shortest path between $u$ and $v$ is unique in $P$ and is denoted as $\Pi(u,v)$.
	Observe that all non-terminal vertices of a shortest path are non-convex \cite{g-vap-07}.
	We will assume an implicit ordering of vertices on $\Pi(u,v)$ from $u$ to $v$. 
	When we say that some vertex $w$ is the first (or last) 
	vertex on $\Pi(u,v)$ with a certain property, we mean that $w$ precedes 
	(respectively, succeeds) all other vertices with that property on $\Pi(u,v)$.
	
	For a proof of Theorem~\ref{colthm}, we have got the following sequence of claims.
	Consider, in all of them, a $K_5$-free reduced polygon $P$
	and its three vertices $t_1,t_2,t_3$ forming a triangle~$T\subseteq G$.
	Assume that $T$ is already coloured (which is unique up to a permutation of the colours).
	Suppose that $v_i$ is an uncoloured vertex, such that an edge incident to $v_i$ intersects $T$.
	Then we have the following lemmas.
	
	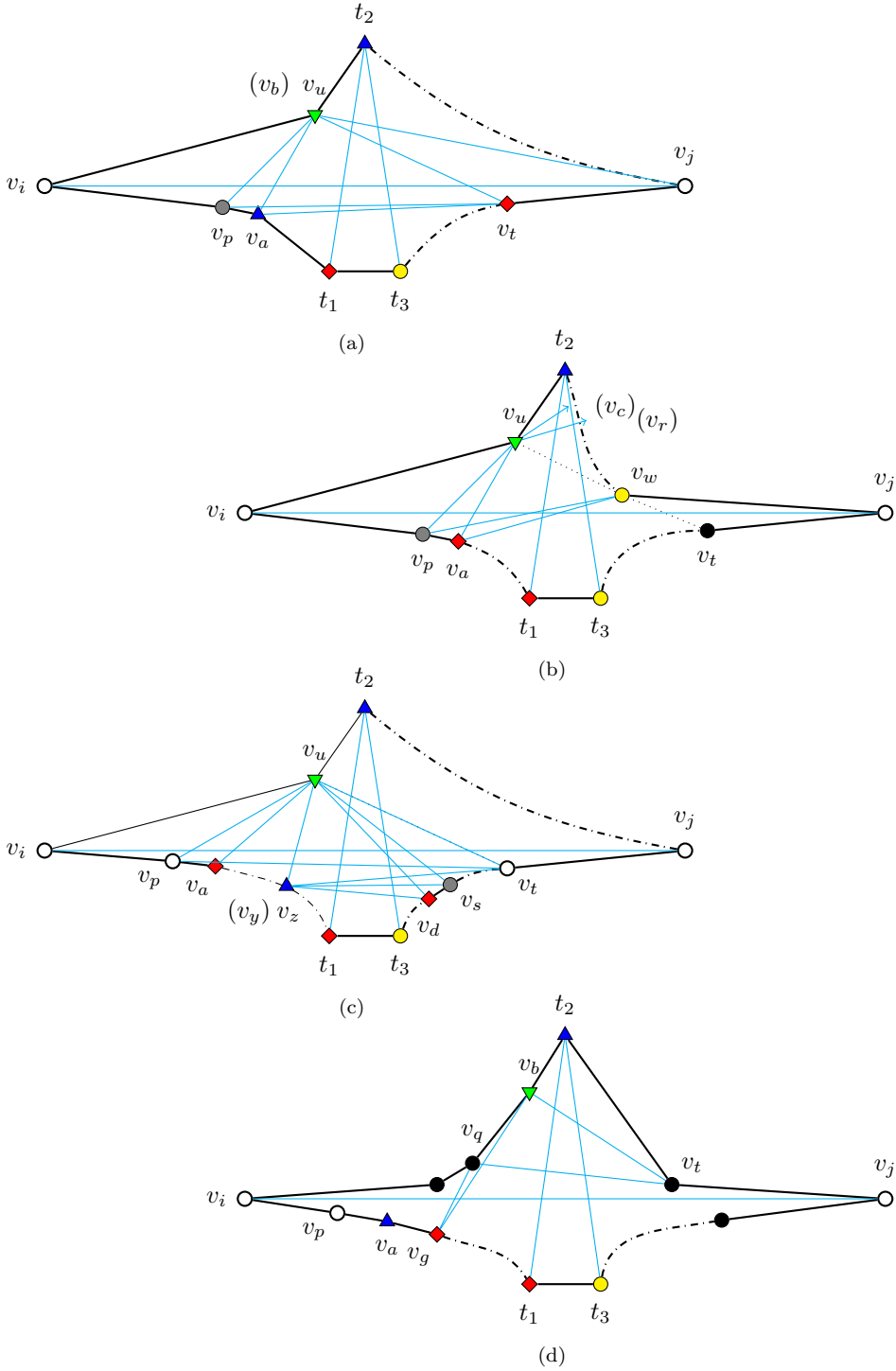
\begin{figure}[htbp]
		\subfloat[]{
			\begin{tikzpicture}[
			diamondd/.style = {inner sep=1.6pt, diamond},
			triangle/.style = {inner sep=1.2pt, regular polygon, regular polygon sides=3},
			dtriangle/.style = {inner sep=1.2pt, regular polygon, regular polygon sides=3, rotate=180}
			]
			\tikzstyle{every node}=[draw, shape=circle, minimum size=3pt, inner sep=2pt];
			\node[thick,fill=white, label=left:$v_i$] (vi) at (0,0) {};
			\node[thick,fill=white, label=$v_j$] (vj) at (9,0) {};
			\node[fill=gray, label=below:$v_p$] (vp) at (2.5, -0.3) {};
			\node[fill=blue,triangle, label=below:$v_a$] (va) at (3,-0.4) {};
			\node[fill=red,diamondd, label=below:$t_1$] (t1) at (4,-1.2) {};
			\node[fill=yellow, label=below:$t_3$] (t3) at (5,-1.2) {};
			\node[fill=red,diamondd, label=below:$v_t$] (vt) at (6.5,-0.25) {};
			\node[fill=blue,triangle, label=$t_2$] (t2) at (4.5,2) {};
			\node[fill=green,dtriangle, label=below:\hbox to
			   0pt{\hspace{-5ex}($v_b$)\hfill}$v_u$] (vu) at  (3.8,1) {};
			
			\draw[thick] (vi)--(vp)--(va)--(t1)--(t3);
			\draw[dash dot,thick] (t3)  to[out=50,in=190] (vt);
			\draw[thick] (vt)--(vj);
			\draw[dash dot,thick] (vj) to[out=170,in=320] (t2);
			\draw[thick] (t2)--(vu)--(vi);
			
			\tikzstyle{every path}=[color=cyan];
			\draw (vi)--(vj);
			
			\draw (vt)--(vp);
			\draw (vt)--(va);
			\draw (vj)--(vu);
			
			\draw (vu)--(vp);
			\draw (vu)--(vt);
			\draw (vu)--(va);
			
			\draw (t2)--(t1);
			\draw (t2)--(t3);
			\end{tikzpicture}
		}
		
		\vspace*{-6ex}\hfill\subfloat[]{
			\begin{tikzpicture}[
			diamondd/.style = {inner sep=1.6pt, diamond},
			triangle/.style = {inner sep=1.2pt, regular polygon, regular polygon sides=3},
			dtriangle/.style = {inner sep=1.2pt, regular polygon, regular polygon sides=3, rotate=180}
			]
			\tikzstyle{every node}=[draw, shape=circle, minimum size=3pt, inner sep=2pt];
			\node[thick,fill=white, label=left:$v_i$] (vi) at (0,0) {};
			\node[thick,fill=white, label=$v_j$] (vj) at (9,0) {};
			\node[fill=gray, label=below:$v_p$] (vp) at (2.5, -0.3) {};
			\node[fill=red,diamondd, label=below:$v_a$] (va) at (3,-0.4) {};
			\node[fill=red,diamondd, label=below:$t_1$] (t1) at (4,-1.2) {};
			\node[fill=yellow, label=below:$t_3$] (t3) at (5,-1.2) {};
			\node[fill=black, label=below:$v_t$] (vt) at (6.5,-0.25) {};
			\node[fill=yellow, label=30:$v_w$] (vw) at (5.3,0.25) {};
			\node[fill=blue,triangle, label=$t_2$] (t2) at (4.5,2) {};
			\node[fill=green,dtriangle, label=below:$v_u$] (vu) at  (3.8,1) {};
			
			\draw[thick] (vi)--(vp)--(va);
			\draw[dash dot,thick] (va) to[out=340,in=120] (t1);
			\draw[thick] (t1)--(t3);
			\draw[dash dot,thick] (t3)  to[out=80,in=180] (vt);
			\draw[thick] (vt)--(vj)--(vw) (t2)--(vu)--(vi);
			\draw[dash dot,thick] (t2)  to[out=290,in=140] (vw);
			
			\draw[dotted] (vt)--(vw)--(vu);
			\tikzstyle{every path}=[color=cyan];
			\draw[->] (vu)--(4.55,1.5)
			  node[draw=none,label=right:~\color{black}($v_c$)] {};
			\draw[->] (vu)--(4.8,1.3)
			  node[draw=none,label=right:~~~~\color{black}($v_r$)] {};
			
			\draw (vi)--(vj);
			
			\draw (vw)--(vp);
			\draw (vw)--(va);
			
			\draw (vu)--(vp);
			\draw (vu)--(va);
			
			\draw (t2)--(t1);
			\draw (t2)--(t3);
			\end{tikzpicture}
		}
		
		\vspace*{-5ex}\subfloat[]{
			\begin{tikzpicture}[
			diamondd/.style = {inner sep=1.6pt, diamond},
			triangle/.style = {inner sep=1.2pt, regular polygon, regular polygon sides=3},
			dtriangle/.style = {inner sep=1.2pt, regular polygon, regular polygon sides=3, rotate=180}
			]
			\tikzstyle{every node}=[draw, shape=circle, minimum size=3pt, inner sep=2pt];
			\node[thick,fill=white, label=left:$v_i$] (vi) at (0,0) {};
			\node[thick,fill=white, label=$v_j$] (vj) at (9,0) {};
			\node[thick,fill=white, label=200:$v_p$] (vp) at (1.8,-0.15) {};
			\node[fill=red,diamondd, label=250:$v_a$] (va) at (2.4,-0.22) {};
			\node[fill=blue,triangle, label=below:\hbox to          
                           0pt{\hspace{-5ex}($v_y$)\hfill}$\!v_z\!\!$] (vz) at (3.4,-0.5) {};
			\node[fill=red,diamondd, label=below:$t_1$] (t1) at (4,-1.2) {};
			\node[fill=yellow, label=below:$t_3$] (t3) at (5,-1.2) {};
			\node[fill=red,diamondd, label=below:$v_d$] (vd) at (5.4,-0.68) {};
			\node[fill=gray, label=-30:$v_s$] (vs) at (5.7,-0.48) {};
			\node[thick,fill=white, label=-30:$v_t$] (vt) at (6.5,-0.25) {};
			\node[fill=blue,triangle, label=$t_2$] (t2) at (4.5,2) {};
			\node[fill=green,dtriangle, label=below:$v_u$] (vu) at  (3.8,1) {};
			
			\draw[dash dot] (va) to[out=345,in=110] (t1);
			\draw[thick] (vi)--(vp)--(va);
			
			\draw[thick] (t1)--(t3);
			\draw[dash dot,thick] (t3)  to[out=70,in=210] (vd);
			\draw[thick] (vd)--(vs);
			\draw[dash dot,thick] (vs)  to[out=30,in=185] (vt);
			\draw[thick] (vt)--(vj);
			\draw[dash dot,thick] (vj)  to[out=170,in=320] (t2);
			\draw (t2)--(vu)--(vi);
			
			\draw[dotted] (vt)--(vu);
			\tikzstyle{every path}=[color=cyan];
			
			\draw (vi)--(vj);
			\draw (vp)--(vt);
			
			\draw (vu)--(vp);
			\draw (vu)--(va);
			\draw (vu)--(vz);
			\draw (vu)--(vt);
			\draw (vu)--(vs);
			\draw (vu)--(vd);
			
			\draw (vz)--(vd);
			\draw (vz)--(vs);
			\draw (vz)--(vt);
			
			\draw (t2)--(t1);
			\draw (t2)--(t3);
			\end{tikzpicture}
		}
		
		\vspace*{-6ex}\hfill\subfloat[]{
			\begin{tikzpicture}[
			diamondd/.style = {inner sep=1.6pt, diamond},
			triangle/.style = {inner sep=1.2pt, regular polygon, regular polygon sides=3},
			dtriangle/.style = {inner sep=1.2pt, regular polygon, regular polygon sides=3, rotate=180}
			]
			\tikzstyle{every node}=[draw, shape=circle, minimum size=3pt, inner sep=2pt];
			\node[thick,fill=white, label=left:$v_i$] (vi) at (0,0) {};
			\node[thick,fill=white, label=$v_j$] (vj) at (9,0) {};
			\node[thick,fill=white, label=200:$v_p$] (vp) at (1.3, -0.2) {};
			\node[fill=blue,triangle, label=270:$v_a$] (va) at
			(2,-0.32) {};
			\node[fill=red,diamondd, label=250:$v_g$] (vg) at
			(2.7,-0.5) {};
			\node[fill=red,diamondd, label=below:$t_1$] (t1) at (4,-1.2) {};
			\node[fill=yellow, label=below:$t_3$] (t3) at (5,-1.2) {};
			\node[fill=black] (vt) at (6.7,-0.3) {};
			\node[fill=black, label=40:$v_t$] (vt') at (6,0.2) {};
			\node[fill=blue,triangle, label=$t_2$] (t2) at (4.5,2.3) {};
			\node[fill=green,dtriangle, label=below:$v_b$] (vb) at  (4,1.5) {};
			\node[fill=black, label=$v_q$] (vq) at  (3.2,0.5) {};
			\node[fill=black] (vq') at  (2.7,0.2) {};
			
			\draw[thick] (vi)--(vp)--(va)--(vg);
			\draw[thick] (t1)--(t3);
			\draw[dash dot,thick] (vg) to[out=340,in=120] (t1);
			\draw[dash dot,thick] (t3)  to[out=80,in=190] (vt);
			\draw[thick] (vt)--(vj)--(vt')--(t2);
			\draw[thick] (t2)--(vb)--(vq)--(vq')--(vi);
			
			\tikzstyle{every path}=[color=cyan];
			\draw (vi)--(vj);
			\draw (t1)--(t2)--(t3);
			\draw (vb)--(vg)--(vq);
			\draw (vb)--(vt')--(vq);

			\end{tikzpicture}
		}
		\caption{Illustration of the proof of Lemma~\ref{mainlem}:
			The vertices with undetermined colours are drawn with white circles.
			The vertices whose colours shall be uniquely determined next, are
			now drawn with gray circles. 
			\\
			(a) $v_p$ forms a $K_4$ with $v_a$, $v_t$ and $v_u$. (b) $v_p$ forms a $K_4$ with $v_a$, $v_u$ and $v_w$.
			(c) $v_s$ forms a $K_4$ with $v_u$, $v_d$ and $v_z$. (d) $v_g$, $v_q$ and $v_b$ form a $K_3$.}
		\label{fig:cases}
	\end{figure}
	
	\begin{lemma} \label{mainlem}
		Assume that two vertices $v_i \in \Gamma(t_1, t_2)$ and $v_j \in \Gamma(t_2, t_3)$ see each other, and the edge $v_iv_j$
		intersects $t_1t_2$ and $t_2 t_3$. Then the colours of all vertices on the four paths
		$\Pi(t_1, v_i)$, $\Pi(t_2, v_i)$, $\Pi(t_2, v_j)$ and $\Pi(t_3, v_j)$, 
		including $v_i,v_j$
		themselves, are uniquely determined by the colours of $T$.
	\end{lemma}
	
	\begin{proof}
		We prove the claim by induction on the four paths.
		As the base case, the first vertices of these paths are the vertices
		of $T$, which are already assigned different colours.
		
		For the induction step, assume that $\Pi(t_1,v_i)$, $\Pi(t_2,v_i)$,
		$\Pi(t_2,v_j)$ and $\Pi(t_3,v_j)$ have been coloured till vertices $v_a$, $v_b$, $v_c$ and
		$v_d$ respectively.
		Also, their immediate uncoloured successors on $\Pi(t_1,v_i)$, $\Pi(t_2,v_i)$,
		$\Pi(t_2,v_j)$ and $\Pi(t_3,v_j)$ are $v_p$, $v_q$, $v_r$ and $v_s$ respectively.
		We aim to show that the colours of at least one of $v_p$, $v_q$, $v_r$ and $v_s$
		is uniquely determined by the already coloured vertices.
		
		We have the following cases (cf.~Figure~\ref{fig:cases}).
		
		\paragraph*{Case 1: $v_p$ sees $v_b$ or some predecessor of $v_b$ on $\Pi(t_2,v_i)$}
		By definition, $v_p$ is the immediate successor of $v_a$ on $\Pi(t_1,v_i)$,
		so $v_p$ must see $v_a$. The right tangent of $v_a$ to $\Pi(t_2,v_i)$ lies to the right of the
		right tangent of $v_p$ to $\Pi(t_2,v_i)$. So, if the right tangent of $v_p$ to
		$\Pi(t_2,v_i)$
		touches $\Pi(t_2,v_i)$ at a vertex $v_u$, then $v_a$ sees $v_u$. Note that either $v_u = v_b$
		or $v_u$ precedes $v_b$ on $\Pi(t_2,v_i)$. In any case, $v_u$ is already coloured.
		Since $v_p$, $v_a$ and $\Pi(t_3,v_j)$ lie on the same side of $v_iv_j$, and $v_p$ is nearer
		to $v_iv_j$ than $v_a$ is, $v_p$ and $v_a$ see a vertex $v_t$ of $\Pi(t_3,v_j)$.
		If $v_u$ also sees $v_t$, and $v_t$ is already coloured, then the claim is proved
		(Figure~\ref{fig:cases}(a)).
		So we consider the other two cases, namely,
		that $v_u$ does not see $v_t$, or that $v_t$ is not yet coloured.
		
		\paragraph{Subcase 1.a:  $v_u$ does not see $v_t$} 
		Since $v_t$ and $v_u$ lie on different sides of
		$\overline{v_iv_j}$ and of $\overline{t_2t_3}$,
		some vertex of $\Pi(t_2,v_j)$ must be blocking $v_u$ and $v_t$.
		Let $v_w$ be the first vertex of $\Pi(t_2,v_j)$ blocking $v_u$ and $v_t$.
		Then $v_u$ sees $v_w$. The vertex $v_w$ is closer to $\overline{v_iv_j}$ 
		than $v_u$ is. Also, $v_w$ lies to the right of $\overrightarrow{v_av_u}$ and $\overrightarrow{v_pv_u}$,
		and to the left of $\overrightarrow{v_av_t}$ and $\overrightarrow{v_pv_t}$. 
		Then the only possible blockers between $v_w$ and $v_p$ or $v_a$ can be from $\Pi(t_2,v_i)$.
		But all the vertices on $\Pi(t_2,v_i)$ preceding $v_u$ are farther from $v_iv_j$ than $v_u$ is.
		So, there can be no such blocker, and $v_w$ must be visible from both $v_a$ and $v_p$ (Figure~\ref{fig:cases}(b)).
		If $v_w$ is already coloured, then the claim is proved. 
		Suppose that $v_w$ is not yet coloured.
		Then consider $v_r$, which now precedes $v_w$ on $\Pi(t_2,v_j)$. The vertices $v_r$ and $v_c$ are 
		consecutive on $\Pi(t_2,v_j)$ and hence see each other.
		Since $\Pi(t_2,v_j)$ and $\Pi(t_1,v_i)$ are on opposite sides of $v_iv_j$, the vertices $v_c$ and $v_r$  
		both see $v_a$ or some vertex preceding $v_a$ on $\Pi(t_1,v_i)$. Let $v_x$ be the last coloured vertex
		of $\Pi(t_1,v_i)$ seen by both $v_c$ and $v_r$. If $v_x \neq v_a$ then let $v_y$ be the last vertex of 
		$\Pi(t_2,v_i)$ that blocks $v_c$ from the successor of $v_y$ on $\Pi(t_1,v_i)$. Then $v_y$ must
		be visible from $v_x$, $v_r$ and $v_c$. Since $v_x$ precedes $v_a$ on $\Pi(t_1,v_i)$, and $v_y$ precedes $v_b$
		on $\Pi(t_2,v_i)$, both $v_x$ and $v_y$ must be already coloured. So, $T$ uniquely determines the colour of $v_r$.
		If $v_x=v_a$ then since $v_u$ is on the right tangent of $v_a$ to $\Pi(t_2,v_i)$, both $v_c$ and $v_r$ see $v_u$.
		Hence, $T$ uniquely determines the colour of $v_r$. Now we move to the
		second subcase.
		
		\paragraph{Subcase 1.b: $v_u$ sees $v_t$, but $v_t$ is not yet coloured}
		Since $v_u$ sees $v_t$, $\Pi(t_2,v_j)$ is a concave chain and the edge $t_1t_3$ exists in $P$,
		$v_u$ must see every predecessor of $v_t$ on $\Pi(t_2,v_j)$.
		This means that both $v_d$ and $v_s$ see $v_u$ (Figure~\ref{fig:cases}(c)). 
		Let the right tangent from $v_d$ touches $\Pi(t_1,v_i)$ in a vertex~$v_y$.
		Then $v_s$ must see $v_y$, because the last vertices $v_i$ and $v_j$ of
		concave chains $\Pi(t_1,v_i)$ and $\Pi(t_3,v_j)$ see each  other. 
		Also, the left tangent of $v_u$ to $\Pi(t_1,v_i)$ must touch $\Pi(t_1,v_i)$ at a vertex equal to or preceding $v_y$.
		Thus, all three of $v_s$, $v_d$ and $v_u$ see a common vertex $v_z$ on
		$\Pi(t_1,v_i)$ which precedes $v_a$,
		since $v_u$ and $v_t$ see $v_a$. Thus, $v_z$ is already coloured, and $v_u$, $v_d$ and $v_z$ form a $K_4$ with $v_s$
		and uniquely determine the colour of $v_s$.

		\paragraph{Case 2: $v_p$ does not see $v_b$ or any predecessor of $v_b$ on $\Pi(t_2,v_i)$}
		Since $\Pi(t_2,v_i)$ is a concave chain, this means that the tangent drawn from $v_p$ to $\Pi(t_2,v_i)$ in the direction
		of $t_2$, has whole $\Pi(t_2,v_b)$ to its left (refer to Figure~\ref{fig:cases}(d)).
		Suppose that $v_b$ does not see $v_a$ or some other vertex of $\Pi(t_1,v_a)$.
		Since also $\Pi(t_1,v_i)$ is a concave chain, and $t_1$ sees $t_2$,
		all blockers between $v_q$ and $\Pi(t_1,v_a)$ must come from $\Pi(v_p,v_i)$, and must include $v_p$.
		But then, the aforementioned tangent drawn from $v_p$ to $\Pi(t_2,v_i)$
		must have at least part of $\Pi(t_2,v_b)$ to the right, which is absurd.

		So, $v_b$ must see $v_a$ or some other vertex of $\Pi(t_1,v_a)$.
		Let $v_g$ denote the last vertex of $\Pi(t_1,v_i)$ seen by $v_b$.
		Then $v_g$ exists, it belongs to $\Pi(t_1,v_a)$ since
		$v_p$ (the successor of $v_a$) does not see~$v_b$, 
		and $v_g$ is seen by~$v_q$ (Figure \ref{fig:cases}(d)).
		Since the vertex $v_g$ is on $\Pi(t_1,v_a)$, it is already coloured.

		Let us now similarly consider a vertex, say $v_t$ on $\Pi(t_2, v_j)$,
		which is seen by both $v_b$ and $v_q$.
		Suppose that $v_g$ or other common coloured neighbour of $v_b$ and $v_q$ sees $v_t$.
		Then we are immediately done if $v_t$ is coloured, or we are
		in Subcase 1.b if $v_t$ is uncoloured.
		Otherwise, some vertex on $\Pi(t_3, v_j)$ blocks all visibilities between $v_t$
		and all the common neighbours of $v_q$ and $v_b$. 
		Then we finish as in Subcase 1.a.
	\end{proof}
	
	\begin{corollary} \label{maincor}
		If any vertex $v_i$ of $P$ sees a vertex of $T$ and their visibility edge crosses one of the edges of $T$,
		then the colour of $v_i$ is uniquely determined by the colours of $T$.
	\end{corollary}
	\begin{proof}
		Without loss of generality, suppose that $v_i$ sees $t_1$, and $v_it_1$ crosses $t_2t_3$. Then $v_j=t_1$, $\Pi(t_2,v_j) = t_2t_1$
		and $\Pi(t_1,v_j) =t_1$, and Lemma~\ref{mainlem} proves the claim.
	\end{proof}

	\begin{figure}
		\centering
		\begin{tikzpicture}[scale=3,
			diamondd/.style = {inner sep=1.6pt, diamond},
			triangle/.style = {inner sep=1.2pt, regular polygon, regular polygon sides=3},
			dtriangle/.style = {inner sep=1.2pt, regular polygon, regular polygon sides=3, rotate=180}
		]
		\tikzstyle{every node}=[draw, shape=circle, minimum size=3pt, inner sep=2pt];
		\node[fill=red,diamondd, label=left:$t_3$] (t1) at (0,-0.1) {};
		\node[fill=blue,triangle, label=$t_1$] (t2) at (0.5, 0.7) {};
		\node[fill=yellow, label=right:$t_2$] (t3) at (1, -0.1) {};
		\node[fill=blue,triangle, label=$v_a$] (va) at (1.7,0.53) {};
		\node[fill=gray, label=$v_u$\hbox to 0pt{~($v_z$)\hspace*{-2em}}] (vu) at (2.1,0.85) {};
		\node[fill=green,dtriangle, label=below:$v_b$] (vb) at
		(2.9,0.55) {};
		\node[fill=yellow, label=right:$v_w$] (vd) at (2.9,0.45) {};
		\node[fill=blue,triangle, label=$v_i$] (vi) at (3.7,0.6) {};
		
		\draw[thick] (t2)--(t1)--(t3);
		\draw (vu)--(va);
		\draw (vu)--(vb);
		\draw (va)--(vb)--(vd);
		\draw (va)--(vd);
		\draw (vb)--(vd);
		\draw[dash dot] (t2) to[out=340,in=180] (va);
		\draw[dash dot] (vb)--(vi);	
		\draw[dash dot] (t3) to[out=20,in=190] (vi);	
		\tikzstyle{every path}=[color=cyan];
		\draw[thick] (0.4,0.2)--(vi);
		\draw[thick] (t2)--(t3);
		\draw[thick] (vu)--(vd);
		\end{tikzpicture}
		
		\begin{tikzpicture}[scale=3,
			diamondd/.style = {inner sep=1.6pt, diamond},
			triangle/.style = {inner sep=1.2pt, regular polygon, regular polygon sides=3},
			dtriangle/.style = {inner sep=1.2pt, regular polygon, regular polygon sides=3, rotate=180}
		]
		\tikzstyle{every node}=[draw, shape=circle, minimum size=3pt, inner sep=2pt];
		\node[fill=red,diamondd, label=left:$t_3$] (t1) at (0,-0.1) {};
		\node[fill=blue,triangle, label=$t_1$] (t2) at (0.5, 0.7) {};
		\node[fill=yellow, label=right:$t_2$] (t3) at (1, -0.1) {};
		\node[fill=blue,triangle, label=$v_a$] (va) at (1.7,0.53) {};
		\node[fill=gray, label=$v_u$] (vu) at (1.9,0.85) {};
		\node[fill=green,dtriangle, label=below:$v_b$] (vb) at (2.9,0.55) {};
		\node[fill=yellow, label=below:$v_c$] (vc) at (1.3,0.06) {};
		\node[fill=red,diamondd, label=below:$v_d$] (vd) at (3.2,0.5) {};
		\node[fill=white] (vbl) at  (2.7,0) {};
		\node[fill=blue,triangle, label=$v_i$] (vi) at (3.7,0.6) {};
		
		\draw[thick] (t2)--(t1)--(t3);
		\draw (vu)--(va);
		\draw (vu)--(vb);
		\draw (va)--(vb)--(vd);
		\draw (va)--(vc)--(vb);
		\draw (va)--(vd);
		\draw (vb)--(vd);
		\draw (vd)--(vc);
		\draw[dash dot] (t2) to[out=340,in=180] (va);
		\draw[dash dot] (vb)--(vi);	
		\draw[dash dot] (t3) to[out=50,in=190] (vc);	
		\draw[dash dot] (vd)--(vi);	
		\tikzstyle{every path}=[color=cyan];
		\draw[thick] (0.4,0.2)--(vi);
		\draw[thick] (t2)--(t3);
		\draw[thick] (vu)--(vbl);
		\end{tikzpicture}
		
		\caption{Illustration of the proof of Theorem~\ref{lemalgo}.
		Top: an edge incident to $v_u$ sees a coloured vertex on $\Pi(t_2, v_i)$.
		Bottom: an edge incident to $v_u$ crosses an edge $v_cv_d$ of $\Pi(t_2, v_i)$.
		Both cases can be resolved by an application of Lemma
		\ref{mainlem} and Corollary \ref{maincor} to the already
		coloured triangles $v_av_bv_w$ and $v_av_bv_c$, respectively.}
		\label{fig:algo}
	\end{figure}
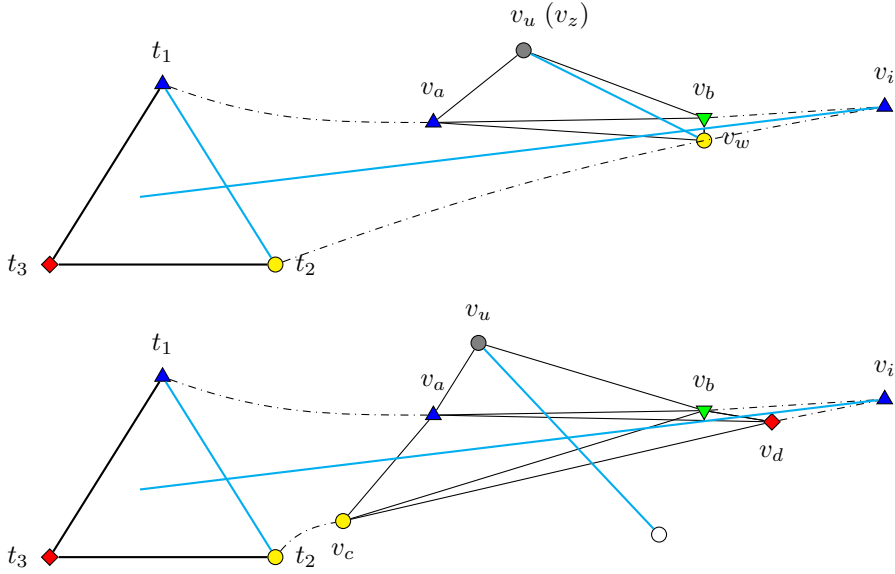

	\begin{theorem} \label{lemalgo}
		If a reduced polygon is 4-colourable, then it has a unique
		4-colour\-ing up to a permutation of colours.
	\end{theorem}
	\begin{proof}
		Consider a triangle $T$ in a reduced polygon $P$. 
		If $P$ is not just $T$, then at least one edge of $T$ is not a boundary edge of $P$.
		Without loss of generality, let $t_1t_2$ be such an edge. Since $P$ is reduced, there must be a vertex
		$v_i$ on the boundary chain $\Gamma(t_1, t_2)$ 
		such that an edge incident to $v_i$ crosses $t_1 t_2$.
		By Lemma \ref{mainlem} and Corollary \ref{maincor}, if $P$ is 4-colourable, then 
		all vertices on the shortest paths $\Pi(t_1, v_i)$ and $\Pi(t_2, v_i)$, including
		$v_i$, have a 4-colouring uniquely determined by $T$.
		In case $t_2 t_3$ or $t_3 t_1$ are not boundary edges of $P$, we can similarly find $v_j$ on
		$\Gamma(t_2, t_3)$ and $v_k$ on $\Gamma(t_3, t_1)$ and uniquely 4-colour $\Pi(t_2, v_j)$,
		$\Pi(t_3, v_j)$, $\Pi(t_3, v_k)$ and $\Pi(t_1, v_k)$. 

		Now, all the remaining uncoloured vertices of $P$ are on
		boundary chains of the form $\Gamma(v_a, v_b)$,
		where $v_a$ and $v_b$ are two consecutive vertices in one of the six paths mentioned above. Furthermore,
		no vertex in the polygonal chain $\Gamma(v_a, v_b)$, other than $v_a$ and $v_b$, is coloured.
		Without loss of generality, let $v_a$ and $v_b$ be two consecutive vertices on $\Pi(t_1, t_2)$.
		If $v_a v_b$ is not a boundary edge of $P$, then since $P$ is reduced, there must be an
		uncoloured vertex $v_u$ in $\Gamma(v_a, v_b)$ such that an edge incident to $v_u$ crosses $v_a v_b$.
		This edge is either incident to a vertex of $\Pi(t_2, v_i)$, or crosses an edge of
		$\Pi(t_2, v_i)$.

		Consider the case where such an edge from $v_u$ to a vertex of $\Pi(t_2, v_i)$ exists.
		Then consider a vertex $v_w$ that is closest to	$\overline{v_av_b}$ among all the vertices
		of $\Pi(t_2, v_i)$ that see an internal vertex (say, $v_z$) of $\Gamma(v_a, v_b)$
		(top of Figure~\ref{fig:algo}). Since the edge $v_wv_z$ exists,
		$v_w$ cannot be blocked by any vertex of $\Pi(t_1, v_i)$. Due to the choice of $v_w$, no vertex of 
		$\Pi(t_2, v_i)$ can block $v_w$ from $v_a$ or $v_b$. So, $v_w$ sees both $v_a$ and $v_b$.
		Then, based on the triangle $v_a v_b v_w$,
		Lemma \ref{mainlem} and Corollary \ref{maincor} can be used to 
		uniquely determine a $4$-colouring for $\Pi (v_a, v_z)$ and $\Pi(v_b,v_z)$.

		Now consider the case in which $v_u$ does not see any vertex of $\Pi(t_2, v_i)$,
		but an edge incident to $v_u$ crosses an edge $v_cv_d$ of 
		$\Pi(t_2, v_i)$, where $v_c$ precedes $v_d$ (bottom of Figure~\ref{fig:algo}).
		Then $v_c$ (as well as $v_d$) must see both $v_a$ and $v_b$,
		since there cannot be any blockers in $\Gamma(v_a, v_b)$ or
		$\Gamma(v_c, v_d)$ which are not on $\Pi(t_1, v_i)$ or $\Pi(t_2, v_i)$.
		Again, based on the triangle $v_a v_b v_c$,
		Lemma \ref{mainlem} can be used to uniquely determine a $4$-colouring 
		for $\Pi (v_a, v_u)$ and $\Pi(v_b,v_u)$.

		Now we recurse the above procedure. Let $T_1 = \{T\}$, and let 
		$S_1 = \{\Pi(t_1, v_i),$ $\Pi(t_2, v_i), \Pi(t_2, v_j), \Pi(t_3, v_j), \Pi(t_3, v_k), \Pi(t_1, v_k)\}$.
		Note that we have assumed that none of the edges of $T$ are boundary edges. 
		If some edges of $T$ are boundary edges then $S_1$ will have less elements.
		By the above procedure, we can uniquely 4-colour all vertices of all
		paths of $S_1$. Then, all the uncoloured vertices $v_u$ lie on $\Gamma(v_a, v_b)$,
		where $v_a$ and $v_b$ are consecutive vertices of some path of $S_1$. 
		For each such $v_a v_b$, we find a new triangle based on $v_a v_b$ as above,
		and two new paths of the form $\Pi (v_a, v_u)$ and $\Pi(v_b, v_u)$.
		Let $T_2$ denote the set of all such new triangles, and $S_2$ denote the set of all
		newly coloured shortest paths obtained this way. 
		In general, following the same method
		we can always construct $T_{i+1}$ and $S_{i+1}$ from $T_i$ and $S_i$, until all vertices of $P$ are coloured. Since in each step, the colours of vertices
		are uniquely determined, it follows that if $P$ has a 4-colouring, then it must be unique.
	\end{proof}
\smallskip

	\subsection{Computing a 4-colouring without polygonal representation}
	
	In the previous section, we have proved that if a reduced polygon is 4-colourable, then its 4-colouring must be unique up to permutations.
	Now we use the property to derive a polynomial time 4-colouring algorithms for the visibility graph of a polygon, even when the polygonal embedding
	or boundary are not given.
	First we need to define a few structures and operations.

	\begin{definition}
		Call a pair of adjacent vertices whose removal disconnects
		a given graph $G$, a \emph{bottleneck pair}.
		Consider removing all the bottleneck pairs from~$G$. We are left with connected components of $G$. 
		Now, consider any bottleneck pair $(x,y)$. Suppose that $x$ and $y$ were earlier adjacent to a set of vertices $S_x$ and $S_y$ of a connected component $C_i$.
		Then create a copy of $(x,y)$ and re-connect them with edges with the vertices of $S_x$ and $S_y$ respectively. Do this with every bottleneck pair of $G$.
		Call the subgraphs of $G$ so formed as \emph{reduced subgraphs} of $G$.
	\end{definition}
	We have the following lemma.

	\begin{lemma} \label{lem:bottleneck}
		Let $G$ be the visibility graph of a polygon~$P$.
		Each bottleneck pair of $G$ corresponds to an internal edge of $P$ that is not intersected by any other internal edge of $P$, and vice versa.
	\end{lemma}
	
	\begin{proof}
		We use the same notations for the vertices of $G$ and their corresponding vertices of $P$.
		Consider any internal edge $xy$ of $P$ such that no other internal edge of $P$ intersects it. 
		Then disconnecting the edge $xy$ and the vertices $x$ and $y$ disconnects $G$.
		So $(x,y)$ is a bottleneck pair.
		Conversely, suppose that $(x,y)$ is a bottleneck pair.
		Then $xy$ is an internal edge of $P$ since deleting a boundary edge does not disconnect~$G$.
		Let $P_1$ and $P_2$ be the two subpolygons of $P$ divided by $xy$.
		If there was a visibility edge from $P_1$ to $P_2$ not
		incident to $x,y$, then since the visibility graphs of $P_1$
		and of $P_2$ are connected, deleting $xy$ would again not disconnect~$G$.
		So, is an internal edge of $P$ not intersected by any other internal edge of $P$.
	\end{proof}
	
	The corollary below follows immediately from Lemma~\ref{lem:bottleneck}.
	
	\begin{corollary}
		Each reduced subgraph of $G$ is the visibility graph of some reduced subpolygon of $P$.
		Likewise, each reduced subpolygon of $P$ has a reduced subgraph of $G$ as its visibility graph.
	\end{corollary}

	\begin{algorithm}[t]
		\caption{4-colourablity of visibility graphs of simple polygons}
		\label{alg:4colouring}
		\hrule\medskip
		\KwIn{A graph $G$ with the promise of being the visibility graph of
			a simple polygon}
		\KwOut{Whether $G$ is $4$-colourable or not. If so, then a proper $4$-colouring of $G$.}
		\smallskip
		Identify all edges $uv$ of $G$, such that removal of $u$ and $v$ disconnects $G$

		Delete all these bottleneck pairs (i.e.,~$u,v$) and partition $G$ into connected components $G_1, G_2, \ldots,
		G_k$. To each connected component of $G$, add copies of the bottleneck pairs
		which were originally attached to it \;

		\ForEach{connected component $G_i$}{
			Locate a triangle in $G_i$ and assign three colours to its vertices\;

			\Repeat{Each vertex in $G_i$ is coloured}
			{
				Locate a vertex adjacent to all 3 vertices of an already coloured triangle in $G_i$\;
			}
		}
		\If{two adjacent vertices receive the same colour}{
			Output `non-4-colourable'\;

			Terminate\;
		}
		Glue the connected components back by merging the corresponding vertices of the two copies of each bottleneck pair\;

		Permute the colours of the vertices so that there is no conflict.
	\end{algorithm}

	Now, in light of the above Algorithm~\ref{alg:4colouring} and Theorem \ref{lemalgo}, we prove Theorem~\ref{colthm}.
	
	\begin{proof} [Proof of Theorem \ref{colthm}]
		Corollary~\ref{maincor} shows that the reduced subgraphs correspond to reduced polygons. 
		By Theorem~\ref{lemalgo} (and its proof), 4-colourable reduced polygons have unique 4-colourings
		which can be found iteratively by colouring each time a vertex with some
		three previously distinctly coloured neighbours. 
		Since the algorithm always chooses a colour for a vertex by
		this iterative scheme, the computed (partial) 4-colouring is the only one possible.
		So, the algorithm is correct.
		
		Let the number of vertices and edges in $G$ be $n$ and $m$ respectively. The
		bottleneck pairs that do not cross any other chord, can be found in $O(m^2)$ time.
		Thus, the decomposition of $P$ into reduced subpolygons takes $O(m^2)$ time.
		A vertex adjacent to every vertex of a coloured triangle can be found in $O(n)$ time.
		While computing the colouring on the shortest paths, a pointer
		can be kept on each of the shortest paths, and the colouring takes $O(n)$ time. The colouring step can be iterated at most once 
		for each vertex, so the complexity for all vertices is $O(n^2)$. Checking for conflict takes $O(m)$ time. Finally, rejoining 
		the reduced subgraphs takes $O(n)$ time. Thus, the complexity of the algorithm is $O(m^2)$.
	\end{proof}

	\section{Hardness of 5-colourability}
	\label{sec:hardness5}
	
	In this section we prove that the problem of deciding whether the visibility
	graph $G$ of a given simple polygon $P$ can be properly coloured with 
	$5$ colours, is NP-complete.
	
	Membership of our problem in NP is trivial (since $G$ can be efficiently
	computed from $P$ and then a colouring checked on~$G$).
	We are going to present a polynomial reduction from the NP-hard problem of
	$3$-colourability of general graphs.
	Our reduction shares some common ideas with reductions on visibility
	graphs presented in \cite{ls-cavg-95,DBLP:conf/fsttcs/CagiriciHR17},
	but the main difference is in not using the SAT problem (which makes our case even simpler).
	The rough outline of the reduction is depicted in
	Figure~\ref{fig:hardness-rough}.

	\begin{figure}[tb]
		$$
		\begin{tikzpicture}[scale=3]
		\tikzstyle{every node}=[minimum size=3pt, inner sep=0pt];
		\draw (2.2,0.25) node {edges of $H$} ;
		\draw (2.2,2.2) node {vertices of $H$} ;
		\draw (2.4,1.97) node {\bf\dots\dots\dots} ;
		\draw (2.4,0.52) node {\bf\dots\dots\dots} ;
		\draw (0.5,0.3) node {$v_1v_4$} ;
		\draw (0.9,0.35) node {$v_2v_6$} ;
		\draw (1.3,0.38) node {$v_4v_8$} ;
		\draw (1.7,0.40) node {$v_7v_{11}$} ;
		\tikzstyle{every path}=[draw,color=lightgray, thick];
		\draw (0,0.5) arc (101:79:11) -- (4.2,2) ;
		\draw (0,0.5) -- (0,2) arc (-101:-79:11) ;
		\tikzstyle{every path}=[draw,thick, color=black];
		\tikzstyle{every node}=[draw,shape=circle,fill=black, minimum size=2pt, inner sep=0pt];
		\draw (0,0.5) -- (0,2)
		-- (0.2,1.96) -- (0.25,2.05) node[label=above:$v_1$] {} -- (0.3,1.95)
		-- (0.4,1.93) -- (0.45,2.02) node[label=above:$v_2$] {} -- (0.5,1.915)
		-- (0.6,1.90) -- (0.65,1.99) node[label=above:$v_3$] {} -- (0.7,1.89)
		-- (0.8,1.875) -- (0.85,1.96) node[label=above:$v_4$] {} -- (0.9,1.865)
		-- (1.0,1.85) -- (1.05,1.94) node[label=above:$v_5$] {} -- (1.1,1.84)
		-- (1.2,1.83) -- (1.25,1.92) node[label=above:$v_6$] {} -- (1.3,1.825)
		-- (1.4,1.82) -- (1.45,1.9) node[label=above:$v_7$] {} -- (1.5,1.815)
		-- (1.6,1.81) -- (1.65,1.89) node[label=above:$v_8$] {} -- (1.7,1.803)
		-- (1.8,1.803) -- (1.85,1.89) node[label=above:$v_9$] {}
		-- (1.9,1.797) -- (2.0,1.797) ;
		\draw (4.2,0.5) -- (4.2,2)
		-- (4.0,1.96) -- (3.95,2.05) node[label=above:$v_n$] {} -- (3.9,1.95)
		-- (3.8,1.93) -- (3.75,2.02) node[label=above:$v_{n-1}$] {} -- (3.7,1.915)
		-- (3.6,1.90) -- (3.55,1.99) node[label=above:$\!\!\!\!\!v_{n-2}$] {}
		-- (3.5,1.89) -- (3.4,1.875) ;
		\draw (0,0.5) -- (0.5,0.58) -- (0.4,0.4) -- (0.6,0.4) -- (0.52,0.585)
		-- (0.9,0.635) -- (0.8,0.45) -- (1.0,0.45) -- (0.92,0.64)
		-- (1.3,0.67) -- (1.2,0.48) -- (1.4,0.48) -- (1.32,0.675)
		-- (1.7,0.692) -- (1.6,0.50) -- (1.8,0.50) -- (1.72,0.695) -- (1.9,0.7) ;
		\draw (4.2,0.5) -- (3.7,0.58) -- (3.8,0.4) -- (3.6,0.4) -- (3.68,0.585)
		-- (3.3,0.635) ;
		\tikzstyle{every path}=[draw, color=cyan, dashed,thick];
		\draw (0.855,1.96) -- (0.51,0.58) -- (0.25,2.05) ;
		\draw (0.445,2.02) -- (0.91,0.635) -- (1.255,1.92) ;
		\draw (0.85,1.96) -- (1.31,0.67) -- (1.655,1.89) ;
		\draw (1.45,1.9) -- (1.71,0.692) -- (2.3,1.87) ;
		\end{tikzpicture}
		$$
		\caption{A scheme of the polygon $P$ constructed from a
		given graph~$H$ in the proof of Theorem~\ref{thm:5hardness}.
		There are two mostly concave chains, top and bottom one.
		The top sawtooth chain features black-marked vertices
		$v_1,v_2,\ldots,v_n$ for each of the $n$ vertices of~$H$.
		The bottom chain contains, for each edge $v_iv_j$ of $H$
		(such as $v_1v_4,v_2,v_6,v_4,v_8,v_7v_{11}$ in the picture), a
		triangular pocket glued to $P$ by a tiny pinhole passage.
		This pocket of $v_iv_j$ is adjusted such that its lower
		corners can see precisely the top vertices $v_i$ and $v_j$,
		respectively (cf.~Figure~\ref{fig:hardness-vdetail}). 
		The important visibility are sketched here with dashed lines.
		Altogether, we can get a proper $5$-colouring of
		the visibility graph of~$P$ if and only if the vertices
		$v_i$ and $v_j$ receive distinct colours for every edge
		$v_iv_j\in E(H)$.}
		\label{fig:hardness-rough}
	\end{figure}
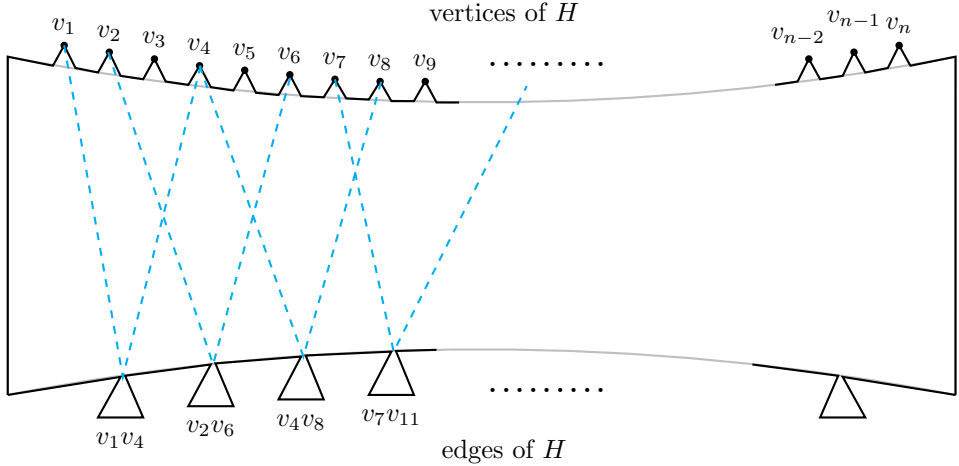

	\begin{theorem}
 	\label{thm:5hardness}
		The problem -- given a simple polygon $P$ in the plane, to decide whether
		the visibility graph of $P$ is properly $k$-colourable --
		is NP-complete for every~$k\geq5$.
	\end{theorem}
	
	\begin{proof}
	As mentioned, the problem is in NP since one can construct the
	visibility graph $G$ of $P$ in polynomial time
	\cite{OROURKE1998105,Hershberger89} and then verify a colouring.
	In the opposite direction, we reduce from the NP-complete problem of
	$3$-colouring a given graph~$H$.

	Let $V(H)=\{v_1,\ldots,v_n\}$.
	The polygon $P$ constructed from $H$ is shaped as in Figure~\ref{fig:hardness-rough}.
	The top chain of $P$ consists of $3n+2$ vertices in a sawtooth
	configuration, such that the convex vertices of the teeth are marked
	by~$v_1,\ldots,v_n$. The picture is scaled such that each $v_i$ sees
	the whole bottom chain.
	The bottom chain contains, for each edge $v_iv_j\in E(H)$, $i<j$ (in an
	arbitrary order of edges), a ``pocket'' consisting of $5$ vertices
	$p_{ij}^1,p_{ij}^2,p_{ij}^3,p_{ij}^4,p_{ij}^5$ in order, as
	detailed in Figure~\ref{fig:hardness-vdetail}.
	Importantly, $p_{ij}^1$ and $p_{ij}^5$ are mutually so close that
	the vertices $p_{ij}^2,p_{ij}^3$ in the lower left corner can see
	only the vertex $v_j$ (of course, besides $p_{ij}^1$ and $p_{ij}^5$)
	and the vertex $p_{ij}^4$ in the lower right corner can see   
        only the vertex $v_i$.
	
	\begin{figure}[tb]
		$$
		\begin{tikzpicture}[xscale=1.6, yscale=1.4]
		\tikzstyle{every path}=[draw, color=black];
		\tikzstyle{every node}=[draw,shape=circle,fill=black, minimum size=2pt, inner sep=0pt];
		\draw (2,0.95) -- (3,1) node[label=200:$\!p_{ij}^1~~~$] {} 
		-- (2.4,0.03) node[label=130:$p_{ij}^2$] (p2) {} 
		-- (2.43,0) node[label=270:$p_{ij}^3$] (p3) {} 
		-- (3.85,0) node[label=right:$p_{ij}^4$] (p4) {}
		-- (3.1,1.01) node[label=350:$~~~p_{ij}^5\!\!\!\!$] {} -- (4,1.05) ;
		\draw (0,2.5) -- (0.7,2.5) 
		-- (1.5,3) node[label=above:$v_i$] {} -- (2.3,2.5) -- (2.5,2.5) ; 
		\draw (3.3,2.53) -- (3.5,2.53) 
		-- (4.35,3.02) node[label=above:$v_j$] {} -- (5,2.55) -- (5.5,2.55) ; 
		\tikzstyle{every path}=[draw, color=cyan, dashed];
		\draw (1.6,3) -- (p4) -- (1.33,3);
		\draw (4.25,3) -- (p2) -- (4.53,3);
		\draw (4.44,3) -- (p3) -- (4.16,3);
		\end{tikzpicture}
		$$
		\caption{A detail (not to scale) of the pocket $v_iv_j$ from Figure~\ref{fig:hardness-rough}.
		Note that $v_j$ and $p_{ij}^4$ see the same four vertices
		$p_{ij}^1,p_{ij}^2,p_{ij}^3,p_{ij}^5$, and so they have to be coloured the same.}
		\label{fig:hardness-vdetail}
	\end{figure}
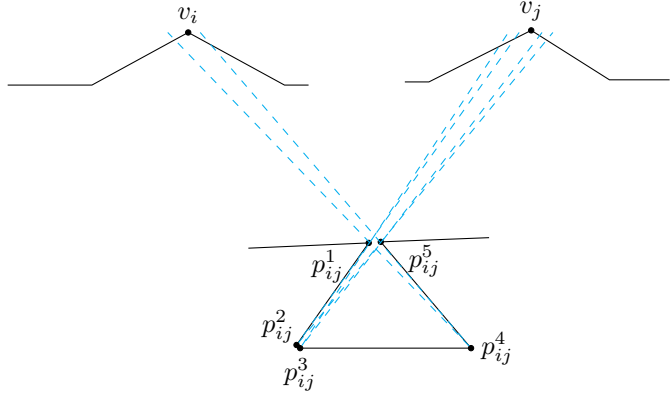

	Assume now that we have got a proper $5$-colouring of the visibility
	graph $G$ of the constructed polygon~$P$.
	We easily argue the following:
	\begin{itemize}
	\item Choose any edge $v_iv_j\in E(H)$. Then the vertices $p_{ij}^1$
	and $p_{ij}^5$ of the corresponding pocket must receive distinct
	colours which we, up to symmetry, denote by $4$ and~$5$.
	Since every vertex of the top chain sees $p_{ij}^1$ and $p_{ij}^5$,
	we get that every vertex $v_k$, $k=1,\ldots,n$, has a colour $1,2$ or~$3$.

	\item For each edge $v_iv_j\in E(H)$, the $5$-tuple of vertices
	$(v_j,p_{ij}^1,p_{ij}^2,p_{ij}^3,p_{ij}^5)$ of $P$ induces a $K_5$,
	and so does the nearly-identical $5$-tuple
	$(p_{ij}^1,p_{ij}^2,p_{ij}^3,p_{ij}^4,p_{ij}^5)$.
	Consequently, in any proper $5$-colouring of~$G$, the vertices $v_j$ 
	and $p_{ij}^4$ get the same colour.
	And since $p_{ij}^4$ sees $v_i$, the colours of $v_i$ and $v_j$ must
	be distinct. 
	\end{itemize}
	Altogether, any proper $5$-colouring of the visibility
        graph $G$ of~$P$ implies a proper $3$-colouring of the graph~$H$.

	On the other hand, assume a proper colouring of the graph $H$ by
	colours $\{1,2,3\}$.
	We give the same colours to the vertices $v_1,\ldots,v_n$ of the top
	chain of~$P$, and we can always complete (e.g., greedily from left to right)
	this partial colouring to a proper $3$-colouring of the top chain of~$P$.
	Then we assign alternate colours $4,5,4,5,\ldots$ to the exposed
	vertices of the bottom chain.
	Finally, we colour the lower corners of the bottom pockets as
	follows; for an edge $v_iv_j\in E(H)$, we give $p_{ij}^4$ the colour
	of $v_j$, and to $p_{ij}^2,p_{ij}^3$ the remaining two colours among $1,2,3$.
	This gives a proper $5$-colouring of the visibility graph $G$ of~$P$.

	The last bit is to show that the construction of $P$ can be realized
	in a grid of polynomial size in~$n=|V(H)|$.
	Both the top and bottom concave shapes can be realized as ``fat'' parabolas,
	requiring only rough resolution of $\Theta(n^2)$ in both
	horizontal and vertical directions.
	This is fully sufficient for the top chain, but realizing the pockets
	of the bottom chain is more delicate.
	Still, fine placement of the pocket of an edge $v_iv_j$ depends
	only on the vertices $v_i$ and $v_j$ of the top chain, and not on
	other pockets. 
	Within the main scale, each pocket has dimensions $\Theta(n)$ and
	the pinhole opening is, say, $\frac1n$, and hence a sufficient precision
	for adjusting the pocket corners is $\Theta(\frac1n)$.
	Altogether, the construction of $P$ is
	achieved on an $\mathcal{O}(n^3)$ grid.
	\end{proof}

	\section{Hardness of 4-colourability with holes}
	\label{sec:polygonswithholes}
	Consider a polygon $P$ together with a collection of pairwise
	disjoint polygons $Q_i$, $i=1,\ldots,k$, such that $Q_i\subseteq int(P)$.
	Then the set $P\setminus int\big(Q_1\cup\ldots\cup Q_k\big)$ is
	called a {\em polygon with holes}.
	In this section we prove that, for polygons with holes, already
	$4$-colourability is an NP-complete problem.
	Given the algorithm for $4$-colouring from Section~\ref{sec1}, it is natural
	that the proof we are going to present should be very different from the
	reduction in Section~\ref{sec:hardness5}.

	For better clarity, we present a construction of a polygon with
	holes as ``digging polygonal corridors in solid mass''.
	These corridors (precisely, their topological closure) will then
	form the point set of our polygon, while the ``mass trapped between''
	corridors will form the holes in the polygon.
	On a high level, our corridors will be composed of elementary
	{\em channels}, as depicted in Figures~\ref{fig:hardness-vertstrip} and
	\ref{fig:hardness-edgestrip}, placed along the lines of a large
	hexagonal (honeycomb) grid in the plane. More details follow next.

	\begin{figure}[tb]
		$$
		\begin{tikzpicture}[scale=1.5]
		\tikzstyle{every node}=[draw,solid,shape=circle,fill=white, inner sep=1.2pt];
		\tikzstyle{every path}=[draw, thick,solid, color=black];
		\draw (1,0) node[label=left:$a_1$] (a){}
		 ++(150:1) node[label=left:$a_2$] (b){}
		 ++(30:1) node[label=left:$a_3$] (c){} ;
		\draw (7,0) node[label=right:$b_1$] (k){}
		 ++(30:1) node[label=right:$b_2$] (l){}
		 ++(150:1) node[label=right:$b_3$] (m){} ;
		\draw (c) -- (4,0.33) node[label=above:$c$] (d){} -- (m) ;
		\draw (a) -- (k) ;
		\begin{scope}[on background layer]
		\tikzstyle{every path}=[draw, line width=3pt, color=gray!30];
		\draw[transform canvas={yshift=-2pt}] (a) -- (k) ;
		\draw[transform canvas={yshift=+2pt}] (c) -- (d) -- (m);
		\end{scope}
		\tikzstyle{every path}=[draw, color=brown, very thick, dotted];
		\draw (a) -- (b) -- (c) -- (a) ;
		\draw (k) -- (l) -- (m) -- (k) ;
		\tikzstyle{every path}=[draw, thin,dashed, color=cyan];
		\draw (a) -- (l) -- (d) -- (b) -- (k) ;
		\draw (a) -- (d) -- (k) ;
		\end{tikzpicture}
		$$ $$
		\begin{tikzpicture}[scale=0.4]
		\path [use as bounding box] (0,-6) rectangle (8,8);
		\def\vertchan#1#2{%
		\tikzstyle{every node}=[draw,solid,shape=circle,fill=white, inner sep=1.2pt];
		\tikzstyle{every path}=[draw, very thick,solid,	color=gray!30];
		\draw (1.08,-0.1) -- (6.92,-0.1) ;
		\draw (1.08,1.08) -- (4,0.42) -- (6.92,1.08) ;
		\tikzstyle{every path}=[draw, thick,solid, color=black];
		\draw (1,0) node (a){} ++(150:1) node[#2] (b){} ++(30:1) node (c){} ;
		\draw (7,0) node (k){} ++(30:1) node[#2] (l){} ++(150:1) node (m){} ;
		\draw (c) -- (4,0.33) node (d){} -- (m) ;
		\draw (a) -- (k) ;
		\tikzstyle{every path}=[draw, color=brown, very thick, dotted];
		\draw (a) -- (b) -- (c) -- (a) ;
		\draw (k) -- (l) -- (m) -- (k) ;
		\tikzstyle{every path}=[draw, thick,solid, color=black];
		#1 ;
		}
		\begin{scope}[transform canvas={xshift=-10mm,yshift=-24.2mm},rotate=60]
			\vertchan{}{} ;
		\end{scope}
		\begin{scope}[transform	canvas={xshift=-13.6mm,yshift=26.3mm},rotate=-60]
			\vertchan{}{} ;
		\end{scope}
		\begin{scope}[transform	canvas={xshift=29.4mm,yshift=-1.45mm},rotate=60]
			\vertchan{}{} ;
		\end{scope}
			\vertchan{\draw (k)--(l)}{fill=black} ;
		\begin{scope}[transform	canvas={xshift=-39.4mm,yshift=-22.7mm}]
			\vertchan{\draw (k)--(l); \draw (b)--(a);
				\tikzstyle{every path}=[draw,dashed, color=brown,thin];
				\draw[->] (0.7,0.5)--(0,1.7);
				}{fill=black} ;
		\end{scope}
		\begin{scope}[transform	canvas={xshift=-39.6mm,yshift=22.8mm}]
			\vertchan{\draw (m)--(l); \draw (b)--(a);
				\tikzstyle{every path}=[draw,dashed, color=brown,thin];
				\draw[->] (0.7,0.5)--(0,1.7);
				}{fill=black} ;
		\end{scope}
		\begin{scope}[transform	canvas={xshift=39.4mm,yshift=22.8mm}]
			\vertchan{\draw (b)--(c);
				\tikzstyle{every path}=[draw,dashed, color=brown,thin];
				\draw[->] (7.3,0.5)--(8,1.7);
				\draw[->] (7.3,0.5)--(8,-0.6)
				}{fill=black} ;
		\end{scope}
		\end{tikzpicture}
		$$
		\caption{An illustration of the proof of Theorem~\ref{thm:4hardness}.
		\\Top: a picture of the {\em vertex channel}, where the
		``solid mass'' remains outside (as indicated by the shade).
		Note that the colour of $c$ is unique in the picture.
		In any proper $4$-colouring,
		$b_1$ must be of the same colour as $a_3$ (since both
		see the triangle $a_1a_2c$), then similarly $b_2$ of the
		same colour as $a_2$ and $b_3$ as~$a_1$.
		\\
		Bottom: How vertex channels are composed by gluing at triangle joins
		along the shape of a hexagonal grid.
		All the black vertices (the {\em flag vertices}) must
		receive the same colour in any proper $4$-colouring.
		}
		\label{fig:hardness-vertstrip}
	\end{figure}
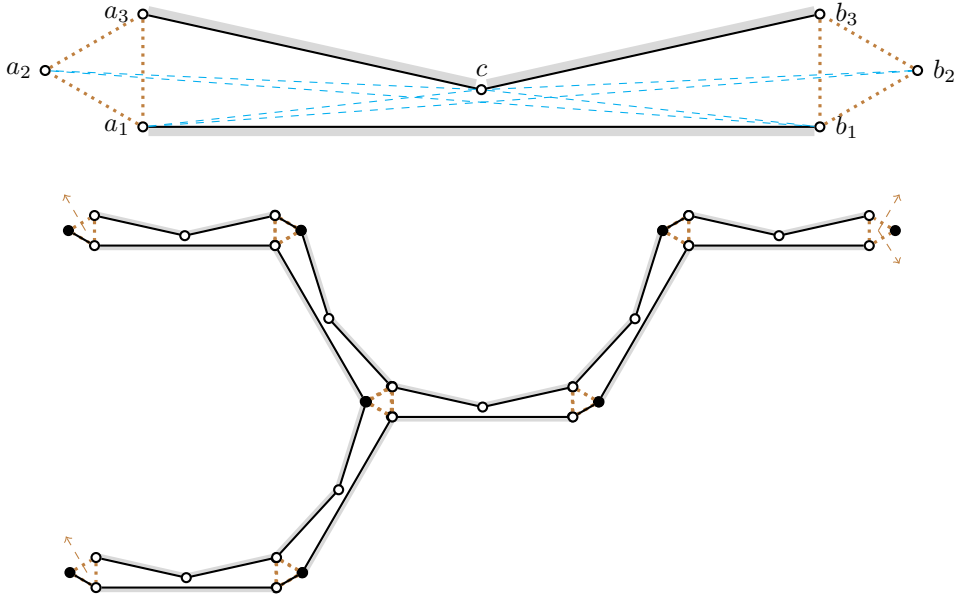

	\begin{figure}[tbp]
		$$
		\begin{tikzpicture}[scale=2.15, rotate=60]
		\tikzstyle{every node}=[draw,solid,shape=circle,fill=white, inner sep=1.2pt];
		\tikzstyle{every path}=[draw, thick,solid, color=black];
		\draw (1,0) node[label=right:$a_1$, fill=black] (a){}
		 ++(150:1) node[label=left:$a_2$] (b){}
		 ++(30:1) node[label=left:$a_3$] (c){} ;
		\draw (7,0) node[label=right:$b_1$] (k){}
		 ++(30:1) node[label=right:$b_2$] (l){}
		 ++(150:1) node[label=left:$b_3$, fill=black] (m){} ;
		\draw (c) -- (2.7,0.39) node[label=left:$c_1$] (d){} 
		 -- (3.0,0.405) node[label=left:$c_2$] (dd){} -- (m) ;
		\draw (a) -- (4.4,0.39) node[label=right:$d_1$] (ee){} 
		 -- (4.4,0.52) node[label=left:$d_2~$] (e){} -- (k) ;
		\begin{scope}[on background layer]
		\tikzstyle{every path}=[draw, line width=3pt, color=gray!30];
		\draw[transform canvas={xshift=-1.7pt,yshift=+1pt}]
			(c) -- (d) -- (dd) -- (m) ;
		\draw[transform canvas={xshift=+1.7pt,yshift=-1pt}]
			(a) -- (ee) -- (e) -- (k) ;
		\end{scope}
		\tikzstyle{every path}=[draw, fill=none, thick, color=black];
		\draw (c) -- (d) -- (dd) -- (m);
		\draw (a) -- (ee) -- (e) -- (k);
		\tikzstyle{every path}=[draw, color=brown, very thick, dotted];
		\draw (a) -- (b) -- (c) -- (a) ;
		\draw (k) -- (l) -- (m) -- (k) ;
		\tikzstyle{every path}=[draw, thin,dashed, color=cyan];
		\draw (dd) -- (a) -- (m) ;
		\draw (b) -- (d) -- (ee) ;
		\draw (dd) -- (e) -- (m) ; \draw (d) -- (a) ;
		\end{tikzpicture}
		~
		\begin{tikzpicture}[scale=0.8, rotate=60,
			triangle/.style = {inner sep=1.2pt, regular polygon, regular polygon sides=3},
			dtriangle/.style = {inner sep=1.2pt, regular polygon, regular polygon sides=3, rotate=180}
		]
		\tikzstyle{every node}=[draw,solid,shape=circle,fill=white, inner sep=1.7pt];
		\tikzstyle{every path}=[draw, thick,solid, color=black];
		\def\onechan#1#2#3#4{%
		\draw (1,0) node[diamond,fill=red] (a){} 
			++(150:1) node[dtriangle,fill=green] (b){} 
			++(30:1) node[triangle,fill=blue] (c){} ;
		\draw (7,0) node[fill=#3] (k){} ++(30:1) node[fill=#4] (l){}
			++(150:1) node[fill=#2] (m){} ;
		\draw (c) -- (2.7,0.42) node[fill=yellow] (d){} 
			-- (3.1,0.43) node[fill=#1] (dd){} -- (m) ;
		\draw (a) -- (4.8,0.2) node[fill=#2] (ee){} 
			-- (4.8,0.5) node[fill=yellow] (e){} -- (k) ;
		\draw[dotted] (a) -- (b) -- (c) -- (a) ;
		\draw[dotted] (k) -- (l) -- (m) -- (k) ;
		}
			\onechan{blue,triangle}{green,dtriangle}{red,diamond}{blue,triangle}
		\begin{scope}[transform canvas={yshift=+30mm}]
			\onechan{blue,triangle}{green,dtriangle}{blue,triangle}{red,diamond}
		\end{scope}
		\begin{scope}[transform canvas={yshift=+60mm}]
			\onechan{green,dtriangle}{blue,triangle}{red,diamond}{green,dtriangle}
		\end{scope}
		\begin{scope}[transform canvas={yshift=+90mm}]
			\onechan{green,dtriangle}{blue,triangle}{green,dtriangle}{red,diamond}
		\end{scope}
		\end{tikzpicture}
		$$
		\caption{Left: a picture of the {\em edge channel}.
		Some important fine details (which cannot be clearly displayed in this scale) are:
		$a_1$ sees $c_2$ and $b_3$, $c_1$ sees $d_1$ but not $d_2$,
		neither of $d_1,c_2$ can see~$a_2$ and neither of $c_1,c_2$ can see~$b_2$.
		Note that in any proper $4$-colouring, $a_1$ and $b_3$ must
		receive distinct colours, while $c_1$ and $d_2$ must have
		the same colour (since they both see the triangle~$a_1c_2d_1$).
		Hence, in particular, the triple of colours used on $a_1a_2a_3$ must be the
		same (up to ordering) as the triple of colours on~$b_1b_2b_3$.
		\\[3pt]
		Right: Examples of proper $4$-colourings of the edge channel.
		Note that the flexibility of these $4$-colourings is not in a
		contradiction with Theorem~\ref{lemalgo} since the chord
		$a_1c_1$ is not crossed by other chords, and likewise the chord~$b_3d_2$.
		}
		\label{fig:hardness-edgestrip}
	\end{figure}
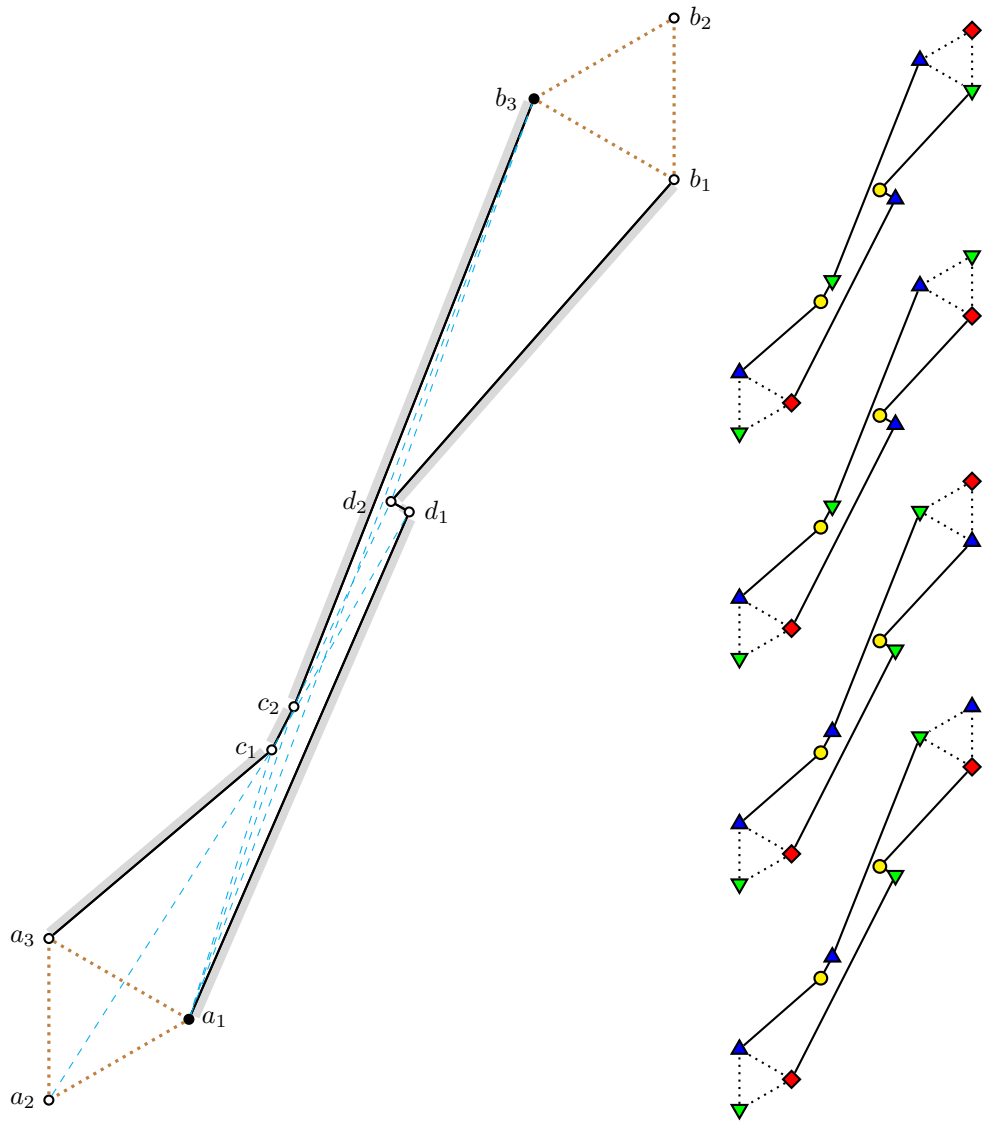
	
	\begin{theorem}\label{thm:4hardness}
		The problem -- given a polygon with holes $P$ in the plane, 
		to decide whether the visibility graph of $P$ is properly $k$-colourable --
		is NP-complete for every~$k\geq4$.
	\end{theorem}
	
	\begin{proof}
	The claim follows from Theorem~\ref{thm:5hardness} for $k\geq5$, and
	so we consider only~$k=4$ here.
	Again, the problem is clearly in NP.
	In the opposite direction, we reduce from the NP-complete problem of
	$3$-colouring a given {\em planar} graph~$H$.

	We first recall a folklore claim that every planar graph $H$ can be
	represented in a usual sufficiently large hexagonal grid in the following way:
	there is a collection of pairwise disjoint subtrees of the grid
	$T_v$: $v\in V(H)$ (representatives of the vertices of~$H$) such that,
	for every edge $uv\in E(H)$, the grid contains an edge between
	$V(T_u)$ and $V(T_v)$ (called a representative edge of~$uv$). 
	(In other words, $H$ is a minor of the grid.)
	To simplify our construction, we may moreover assume that we always
	choose representative edges in the grid which are not of horizontal
	direction (out of the three directions $0^\circ$, $120^\circ$ and~$240^\circ$).

	Having such a representation of the given planar graph $H$ in the
	grid, we continue as follows.
	Let a {\em vertex channel} be the polygonal fragment shown in
	Figure~\ref{fig:hardness-vertstrip}, where the triples $a_1a_2a_3$
	and $b_1b_2b_3$ are the {\em triangle joins} of the channel.
	Channels are composed, after suitable rotation, by gluing their
	triangle joins together, as illustrated in bottom part of 
	Figure~\ref{fig:hardness-vertstrip}.
	When no further channel is glued to a join, then the dotted triangle
	edge(s) is ``sealed'' by a polygon edge.
	Let an {\em edge channel} be the polygonal fragment shown in
        Figure~\ref{fig:hardness-edgestrip}, again having two triangle joins
	$a_1a_2a_3$ and $b_1b_2b_3$ at its ends.
	Edge channels are used and composed in a same way with vertex
	channels, but edge channels cannot be rotated, only mirrored by the
	vertical axis (that is why we do not use them along horizontal grid edges).
	Altogether, we construct a polygon $P$ from $H$ by composing copies of the
	vertex channel along all the grid edges of each $T_v$, $v\in V(H)$,
	and by further composing in copies of the edge channel (possibly mirrored)
	along the representative edges of $H$ in the grid.

	Assume now that we have got a proper $4$-colouring of the visibility graph of $P$.
	For each triangle join, let the vertex with the middle $y$-coordinate
	be called the {\em flag vertex} (it is the vertex which is extreme
	to the left or right).
	One can easily check from Figure~\ref{fig:hardness-vertstrip} that,
	among all vertex channels of one $T_v$, $v\in V(H)$, all the
	triangle joins receive the same unordered triple of colours and, in
	particular, all the flag vertices have the same one colour.
	The same claim can also be derived from Theorem~\ref{lemalgo} applied to
	the standalone simple polygon formed by the vertex channels of $T_v$.
	Furthermore, one can check from Figure~\ref{fig:hardness-edgestrip},
	that also the edge channel maintains the property that both its
	triangle joins must receive the same unordered triple of colours.

	Naturally assuming connectivity of $H$, we hence conclude that every
	triangle join constructed in $P$ receives the same unordered triple
	of colours, say~$\{1,2,3\}$.
	Now, to each vertex $v$ of $H$ we assign the unique colour from
	$\{1,2,3\}$ which occurs on the flag vertices of~$T_v$.
	Since the two flag vertices of the edge channel see each other
	($a_1$ and $b_3$ in Figure~\ref{fig:hardness-edgestrip}), this
	ensures that for every edge $uv\in E(H)$ the colours assigned to $u$
	and $v$ are distinct, and so $H$ is $3$-colourable.

	In the converse direction, we assume that $H$ has a proper $3$-colouring.
	We can routinely $4$-colour the polygonal fragments of each $T_v$, 
	$v\in V(H)$, such that all the flag vertices of $T_v$ get the colour of~$v$.
	Then, for each $uv\in E(H)$ with distinct colours on $u$ and~$v$, 
	we can complete proper $4$-colouring of the fragment of $P$ made by 
	the representative edge channel of $uv$, as shown in the right part
	of Figure~\ref{fig:hardness-edgestrip}.
	Hence the visibility graph of~$P$ is then $4$-colourable.

	Finally, the construction of $P$ is easily done (with negligible
	distortion of the angles of hexagonal grid) within polynomial
	resolution and so in polynomial time.
	\end{proof}

	\section{Conclusions}\label{sec:conclu}
	In this paper we have shown that the problem of deciding 5-colourability for visibility graphs of simple polygons, is NP-complete.
	We have also proved that the 4-colouring problem can be solved for visibility graphs of simple polygons, in polynomial time,
	whereas for visibility graphs of polygons with holes, it becomes NP-complete.
	However, it still remains to be explored whether approximation algorithms
	could exist for the hard colouring problems on visibility graphs of polygons.

	\bibliographystyle{siamplain}
	\bibliography{references}
	
\end{document}